\documentclass[reqno,12pt]{amsart} 
\usepackage{amsmath,amssymb,amsthm,enumerate,mathrsfs,colonequals}
\usepackage{fullpage,url,xspace,setspace}
\usepackage{microtype,verbatim}
\usepackage{calc}
\usepackage[alphabetic,lite,nobysame]{amsrefs} % for bibliography

\usepackage[all,cmtip]{xy}

\usepackage{color}

\usepackage{capt-of}
\usepackage{graphicx}
\usepackage{hyperref}

\DeclareFontFamily{U}{wncy}{}
\DeclareFontShape{U}{wncy}{m}{n}{<->wncyr10}{}
\DeclareSymbolFont{mcy}{U}{wncy}{m}{n}
\DeclareMathSymbol{\Sha}{\mathord}{mcy}{"58}

\DeclareMathOperator{\sgn}{sgn}

\DeclareMathOperator{\SL}{SL}
\DeclareMathOperator{\SLT}{\SL_{2}}
\DeclareMathOperator{\HYP}{HYP}
\DeclareMathOperator{\Quad}{Quad}
\DeclareMathOperator{\Nm}{Nm}

\usepackage{amssymb,amsmath}
\usepackage{amsfonts}
\newcounter{ctfig}

\newcommand{\Z}{{\mathbb Z}}
\newcommand{\Q}{{\mathbb Q}}
\newcommand{\R}{{\mathbb R}}
\newcommand{\B}{\mathcal{B}}

\newcommand{\oV}{\overline{V}}

\newcommand{\OO}{\mathcal{O}}
\newcommand{\E}{\mathcal{E}}

\newcommand{\A}{{\mathbb A}}

\newcommand{\bs}{\beta^{\ast}}
\newcommand{\CC}{\mathbb C}

\DeclareMathOperator{\FP}{FP}

\newcommand{\PP}{[b_1,\ldots , b_N,\overline{a_1,\ldots ,a_k}]}
\newcommand{\p}{(b_1, \ldots , b_N, a_1, \ldots  a_k)}
\theoremstyle{plain}
\newtheorem{thm}{Theorem}[section]
\newtheorem{hyp}[thm]{Hypothesis}

\newtheorem{lemma}[thm]{Lemma}
\newtheorem{prop}[thm]{Proposition}
\newtheorem{cor}[thm]{Corollary}

\theoremstyle{definition}

\newtheorem{remark}[thm]{Remark}
\newtheorem{defn}[thm]{Definition}
\begin{document}
%% \foreach \x in{graphics,floats}{%
%%     \immediate\write18{pdflatex -jobname=template-\x\space "\def\noexpand\placeholder{\x} \noexpand\input{template}"}%
%%     \includepdf[pages=-]{template-\x}%
%% }
\bibliographystyle{plain}
\bibstyle{plain}

\title[Integral points arising from periodic continued fractions]
{The Zariski closure of integral points on varieties parametrizing
periodic continued fractions}

\author{Bruce W. Jordan \and Adam Logan \and Yevgeny Zaytman}

\address{Department of Mathematics, Box B-630, Baruch College,
The City University of New York, One Bernard Baruch Way, New York
NY 10010}
\email{bruce.jordan@baruch.cuny.edu}

\address{The Tutte Institute for Mathematics and Computation,
  P.O. Box 9703, Terminal, Ottawa, ON K1G 3Z4, Canada}
\address{School of Mathematics and Statistics, 4302 Herzberg
  Laboratories, 1125 Colonel By Drive, Ottawa, ON K1S 5B6, Canada}
\email{adam.m.logan@gmail.com}

\address{Center for Communications Research, 805 Bunn Drive,
Princeton, NJ 08540}
\email{ykzaytm@idaccr.org}


\begin{abstract}
Let $R$ be the ring of $S$-integers in a number field $K$.  Let
$\B=\{\beta, \bs\}$ be the multiset of roots of a nonzero quadratic
polynomial over $R$.  There are varieties $V(\B)_{N,k}$ defined over
$R$ parametrizing periodic continued fractions $\PP$ for $\beta$ or
$\bs$.  We study the $R$-points on these varieties, finding
contrasting behavior according to whether groups of units are infinite
or not.  If $R$ is the rational integers or the ring of integers in an
imaginary quadratic field, we prove that the $R$-points of
$V(\B)_{N,k}$ are not Zariski dense. On the other hand, suppose that
$\beta\not\in K\cup\{\infty\}$, $R^\times$ is infinite, and that there
are infinitely many units in the (left) order $R_\beta$ 
of $\beta R+R\subseteq K(\beta)$ with norm to $K$ equal
to $(-1)^k$.  Then we prove that the $R$-points on $V(\B)_{1,k}$ are
Zariski dense for $k\geq 8$ and the $R$-points on $V(\B)_{0,k}$ are
Zariski dense for $k\geq 9$. We also prove that 
$V(\B)_{1,k}$ and $V(\B)_{0,k}$ are $K$-rational 
irreducible varieties
for $k$ sufficiently large.
\end{abstract}

\maketitle

\section{Introduction}
\label{taco}

This work is the third in a series of papers 
(the first two being Brock-Elkies-Jordan \cite{bej} and 
Jordan-Zaytman \cite{jz}) recasting 
periodic continued fractions
as a problem in both diophantine geometry
and arithmetic groups---an attempt at a modern take on a very
classical subject.
Let $K$ be a number field and let $S$ be a finite set of primes of $K$
containing the archimedean ones.  Let $R$ be the ring of $S$-integers
of $K$. To a
a periodic continued fraction $\PP$ with partial
quotients in $R$ (an {\sf  $R$-PCF })
is associated \cite[Defn.~2.4]{bej} a quadratic polynomial
$0\neq Q(x):=\Quad_{N,k}(P)(x)\in R[x]$ with $P$ formally a root
obtained by expanding the equality
\[
\PP=[b_1,\ldots, b_N, a_1,\ldots , a_k,\overline{a_1,\ldots, a_k}].
\]
Now let  $\B=\{\beta, \beta^\ast\}$ be the multiset of roots
of $0\neq Q(x)=Ax^2+Bx+C\in R[x]$.  To this data 
together with a {\sf type} $(N,k)$ comprised of  integers $N\geq 0$, $k\geq 1$ 
\cite[Sect.~3]{bej}
associate an affine {\sf PCF variety} 
$V(Q)_{N,k}=V(\B)_{N,k} \subseteq\A^{N,k}\colonequals
\A^N\times \A^k$ over $R$,
generically of dimension $N+k-2$,
parametrizing periodic continued fractions ($R$-PCFs)
$P=[b_1,\ldots, b_N, \overline{a_1, \ldots, a_k}]$ having 
$\Quad_{N,k}(P)=Q(x)$.
There is a one-to-one correspondence between points
$p=p(P)=\p\in\A^{N,k}$ and PCFs
$P=P_{N,k}(p)=\PP$ of type $(N,k)$.

The variety $V(Q)_{N,k}=V(\B)_{N,k}$ is
 defined using Euler's continuant polynomials;
if 
$P$ is an $R$-PCF
with $\Quad_{N,k}(P)=Q(x)$,
then $p(P)\in V(\B)_{N,k}(R)$.
If $\alpha\in R$ but $\beta=\sqrt{\alpha}\notin R$ (so that
$Q(x)=x^2-\alpha\in R[x]$ is irreducible),  we shorten
the notation to $V(\alpha)_{N,k}:=V(Q)_{N,k}=V(\{\beta,-\beta\})_{N,k}$.
Set $D(a):=\left[\begin{smallmatrix}a & 1\\1 & 0\end{smallmatrix}\right]$
and $t=\left[\begin{smallmatrix}0 & 1\\1 & 0\end{smallmatrix}\right]$.
For the point $p=(b_1, \ldots , b_N, a_1,\ldots ,  a_k)\in V(\B)_{N,k}(R)$, 
the matrix $E:=E_{N,k}(p)=E(P)=E_{N,k}(P)$
from \cite{bej}*{Defn.~2.4} plays a key role:
\begin{equation}
\label{gwen}
E=\begin{bmatrix}E_{11}& E_{12}\\E_{21}& E_{22}\end{bmatrix}=
D(b_1)\cdots D(b_N)D(a_1)\cdots D(a_k)tD(-b_N)\cdots D(-b_1)t.
\end{equation}
We have 
\begin{align}
\label{gwen1}
\nonumber \Quad_{N,k}(P)& =E_{21}x^2+(E_{22}-E_{11})x-E_{12} ,\\
E\begin{pmatrix}\beta\\1\end{pmatrix}& =(E_{21}\beta+E_{22})
\begin{pmatrix}\beta\\1\end{pmatrix},\quad\text{and}\\
\nonumber
E\begin{pmatrix}\beta^\ast\\1\end{pmatrix}&=
(E_{21}\beta^\ast+E_{22})
\begin{pmatrix}\beta^\ast\\1\end{pmatrix}.
\end{align}

In this paper we study the Zariski closure of $V(\B)_{N,k}(R)$
in $V(\B)_{N,k}(\CC)$. We find opposite behavior in cases where
the unit group $R^\times$ is finite or infinite.
In Section \ref{sec:k-eq-q} we consider
$R=\Z$ and prove that the $\Z$-points
are never dense.  We extend this result in Section \ref{sec:k-imag-quad}
to the ring of 
integers in an imaginary quadratic
field.  Finally we consider the case where $R^\times$ is infinite,
$\beta\notin K\cup\{\infty\}$, and 
$N=0,1$.
Let $R_{\beta}\subset K(\beta)$  be the (left) order 
of the $R$-lattice $L_{\beta}=\beta R+R$, cf. Definition \ref{pistol1}.
 Supposing additionally
that there are 
infinitely many units $u\in R_\beta$ with norm $\Nm_{K(\beta)/K}(u) = (-1)^k$,
we prove that  
in $V(\B)_{1,k}(R)$ is Zariski dense for $k\geq 8$.
Assuming the Generalized Riemann Hypothesis as 
in \cite{jz1}*{Sect.~3D}, we prove that $V(\B)_{1,k}(R)$ is
Zariski dense for $k\geq 4$ if $S$ contains a real archimedean prime, 
for $k\geq 5$ if $S$
contains a finite prime, and for $k\geq 6$ in general. 
We prove these results by fibering the variety $V(\B)_{1,k}$ over the 
Fermat-Pell rational curve and then 
analyzing the fibers of this map using
the main results
of \cite{jz}. This also allows us
to determine when the varieties $V(\B)_{1,k}$ and $V(\B)_{0,k}$
are irreducible rational varieties over $K$; see Theorem \ref{queen}
and Remark \ref{pizza}.

\section{The case \texorpdfstring{$R=\Z$}{R=Z}}\label{sec:k-eq-q}

It is  not hard to prove that the $\Z$-points on $V(\B)_{N,k}$ are not Zariski
dense for arbitrary $N$ and $k$ with $N+k>2$.  
This follows from a classical result
on nearest-integer continued fractions \cite{h}*{Sect.~2} ,
which we state using modern notation:

\begin{thm}
\label{Hurwitz}
Every $\alpha\in\R$ has a unique expression in one
of the two forms $[c_1, c_2,\ldots, c_N]$, $[c_1, c_2,\ldots ]$ with all
$c_i\in \Z$ satisfying the 
conditions:
\begin{enumerate}[\upshape (a)]
\item
\label{Hurwitz1}
$|c_i|\geq 2$ for $i>1$.
\item
\label{Hurwitz2}
If $c_i=\pm2$ for $i>1$, then
  $\sgn(c_i)=\sgn(c_{i+1})$.
\item
\label{Hurwitz3}
In case the form is $[c_1, c_2,\ldots c_N]$, then $c_N\neq -2$.
\end{enumerate}
\end{thm}

\begin{cor}
\label{at most 1}
For each $\alpha\in\R , N, k$ there exists at most 
one $\Z$-PCF
\[
\alpha=[c_1,\dots, c_{N},\overline{c_{N+1}, \ldots, c_{N+k}}]
\]
with
$|c_i|>2$ for $1\leq i \leq N+k$.
\end{cor}

Since we only use the uniqueness part of Theorem \ref{Hurwitz}, we
will give a proof below, but first we need a lemma.

\begin{lemma}
\label{interval}
Let $\alpha=[c_1,\ldots ,c_N]$ with $c_i\in\Z$ and possibly $N=\infty$
satisfying the conditions:
\begin{enumerate}[\upshape (a)]
\item
\label{interval1}
$|c_i|\geq 2$.
\item
\label{interval2}
If $c_i=\pm2$ for $i\ne N$, then
  $\sgn(c_i)=\sgn(c_{i+1})$.
\end{enumerate}
Then $1/\alpha \in [-\frac12,\frac12]$ and
$\sgn(\alpha)=\sgn(c_1)$.

\end{lemma}

\begin{proof}
First assume $N$ is finite and proceed by induction on $N$.  If $N=1$
then the result is obvious.  Now suppose the result holds for
$\alpha'=[c_2,\ldots,c_N]$.  Then $\alpha=c_1+1/\alpha'$.  So
if $\lvert c_1 \rvert\ge3$, we have $\lvert c_1+ 1/\alpha'
\rvert\ge 3- 1/2>2$; and if $\lvert c_1 \rvert=2$ we have
$\sgn(c_1)=\sgn(\alpha')=\sgn(1/\alpha')$ so $\lvert
c_1+ 1/\alpha' \rvert>\lvert c_1\rvert=2$.  In either case we
have $\sgn(c_1)=\sgn(\alpha)$.

The case with $N=\infty$ follows by taking limits.
\end{proof}

\begin{proof}[Proof of uniqueness in Theorem \textup{\ref{Hurwitz}}]
Let $\alpha=[c_1, c_2,  c_3,\ldots]=[d_1, d_2, d_3,\ldots]$ be a
counterexample chosen such that the first $i$ with $c_i\ne d_i$ is
minimal among all counterexamples.  Note that we must have $i=1$ since
otherwise $[c_2, c_3, \ldots]=[d_2, d_3,\ldots]$ would be a counterexample
with smaller $i$.  Now let $\alpha'=[c_2, c_3, \ldots]$, so
$\alpha=c_1+1/\alpha'$.  By Lemma \ref{interval},
$1/\alpha'\in[-1/2, 1/2]$.
We cannot  have $1/\alpha'=-1/2$, since then $\alpha'=-2$ and $c_2=-2$,
contradicting condition Theorem \ref{Hurwitz}(\ref{Hurwitz3}).  
Hence $1/\alpha'\in
(-1/2, 1/2]$.
Therefore, with $\lfloor \,\,\,\rceil$ denoting the
nearest integer function, 
$c_1=\lfloor\alpha\rceil$ and similarly $d_1=\lfloor\alpha\rceil$,
contradicting the assumption that $c_1\ne d_1$.
\end{proof}

Prior to proving our first theorem, we recall an 1865 theorem
of Worpitzky:
\begin{thm}
\label{pistol}
\textup{(Worpitzky \cite{wor}, cf. \cite{j-t}*{Cor.~4.36}.)}
Let $C=[c_1,c_2,\ldots , c_n, \ldots]$ be a continued fraction
with $c_i\in\CC$, $i\geq 1$.  If $|c_nc_{n+1}|\geq 4$ for
$n\geq 1$, then $C$ converges.
\end{thm}
%\noindent We note with amusement that Worpitzky was a high-school teacher
%at the renowned \mbox{Friedrichswerdisches} Gymnasium Berlin and Theorem 
%\ref{pistol} was first published in his school's {\em Jahresbericht}.
\begin{thm}
\label{radishes}
Suppose $\B=\{\beta, \beta^{\ast}\}$ is the multiset of
roots of a quadratic polynomial over $\Z$ with $\beta\notin \Q\cup\{\infty\}$.
For any
$N\geq 0$, $k\geq 1$ with $N+k>2$, the $\Z$-points on 
$V(\B)_{N,k}$ are degenerate,
i.e., they are not Zariski dense.
\end{thm}
\begin{proof}
If $N+k>2$, then $\dim V(\B)_{N,k}\geq 1$.
First note that the set 
\[
\mathbf{S}\colonequals \{(c_1,\ldots ,
c_N, \overline{c_{N+1}, \ldots, c_{N+k}})\in V(\B)_{N,k}(\Z)\mid
|c_i|\leq 2\text{ for some $i$, } 1\leq i\leq N+k\}
\]
is not Zariski dense in $V(\B)_{N,k}$.  This is because
$S_{j,k}\colonequals V(\B)_{N,k}\cap (x_j=k)$ is a proper subvariety of 
$V(\B)_{N,k}$ for $1\leq j\leq N+k$ and $k=0,\pm 1,\pm 2$.  Hence 
$S=\cup_{j,k}S_{j,k}$ is a proper subvariety and 
$\mathbf{S}=S(\Z)$.

We now remark that any PCF
\begin{equation}
\label{grisly}
P=[c_1, \ldots, c_{N}, \overline{c_{N+1}, \ldots , c_{N+k}}]
\end{equation}
with $c_i\in\Z$, $|c_i|>2$ for $1\leq i \leq N+k$,
converges.   In fact,
the stronger statement that $P$ as in \eqref{grisly} with 
$c_i\in\Z$, $|c_i|\geq 2$ for $1\leq i\leq N+k$, 
converges follows from Theorem \ref{pistol} above.
By this remark and Corollary \ref{at most 1} we see that
$V(\B)_{N,k}(\Z)\setminus \mathbf{S}$ consists of at most two points,
one corresponding to a PCF converging to $\beta$ and one
corresponding to a PCF converging to $\beta^\ast$.
This proves Theorem
\ref{radishes}.
\end{proof}

\section{Rings of integers in imaginary quadratic fields}\label{sec:k-imag-quad}

There is an analogous result to Theorem \ref{Hurwitz} for
imaginary quadratic fields.  We first establish some notation.

Let $F$ be an imaginary quadratic field and $\OO$ its ring of integers.
The ring $\OO$ is naturally a $2$-dimensional $\Z$-lattice
with $\|z\|^2=z\overline{z}$ for $z\in\OO$.  Let $V_0$ be any closed
fundamental domain for $\OO$ and $U_0$ its interior.  For $c\in\OO$
let $V_c=c+V_0$ and $U_c=c+U_0$ be the corresponding translated
fundamental domains.
%For $a\in\OO$ let $U_a$ be the open $V_a$ be the closed Voronoi cell
%of the lattice $\OO$ centered at $a$:
%\begin{align*}
%U_a = & \{z\in\OO\mid \|z-a\|< \|z-b\| \mbox{  for all $b\neq a \in \OO$},\\
%V_a = & \{z\in\OO\mid \|z-a\|\leq \|z-b\| \mbox{  for all $b\neq a \in \OO$}.
%\end{align*}
  Let
$U_0^{-1}=\{x\in\CC\mid x^{-1}\in U_0\}$ and let
$M=\{c\in\OO\mid V_c\not\subset U_0^{-1}\}$.  Note that $M$ is finite and
that there exists a closed set $W\subset U_0$ such that
$V_c^{-1}\subset W$ for all $c\in\OO\setminus M$.  We
can now state our analogue of Theorem \ref{Hurwitz}.

\begin{thm}
\label{im q}
Every $\alpha\in\CC$ can be expressed in at most one way as
$\alpha=[c_1, c_2, c_3,\dots]$ with $c_i\in\OO\setminus M$ for $i>1$.
\end{thm}

Note that such an expression need not exist.

As in the previous section the proof requires an analogue of Lemma
\ref{interval}.
\begin{lemma}
\label{lem1}
Let $\alpha=[c_1,\ldots,c_N]$ with $c_i\in\OO\setminus M$ and possibly
$N=\infty$. Then \mbox{$1/\alpha\in W\subset U_0$}.
\end{lemma}
\begin{proof}
First assume $N$ is finite and proceed by induction on $N$.  If $N=1$,
then the result follows since $\alpha=c_1\in U_{c_1}\subset U_0^{-1}$.
Now suppose the result holds for $\alpha'=[c_2,\ldots, c_N]$.  Then
$\alpha=c_1+ 1/\alpha'$.  So we
have $\alpha\in c_1+U_0=U_{c_1}\subset W^{-1}$ and $1/\alpha\in
W$.

The case with $N=\infty$ follows by taking limits since $W$ is closed.
\end{proof}

\begin{proof}[Proof of Theorem \textup{\ref{im q}}]
As in the proof of Theorem \ref{Hurwitz} let 
\[
\alpha=[c_1,c_2,c_3,\dots]=[d_1,d_2,d_3,\dots]
\]
 be a counterexample
chosen such that the first $i$ with $c_i\ne d_i$ is minimal among all
counterexamples.  As before we must have $i=1$ since otherwise
$[c_2,c_3,\dots]=[d_2,d_3,\dots]$ would be a counterexample with
smaller $i$.  Now let $\alpha'=[c_2,c_3,\dots]$, so
$\alpha=c_1+1/\alpha'$.  By Lemma \ref{lem1},
$1/\alpha'\in U_0$; therefore $c_1$ is the unique 
element of $\OO$ with $\alpha\in U_{c_1}$ and by the same logic so is $d_1$.
\end{proof}

\begin{thm}
\label{deg im}
Let $\OO$ be the ring of integers in an imaginary quadratic field F.
Suppose $\B=\{\beta,\bs\}$ is the multiset of roots
of $Q(x)=Ax^2+Bx+C\in \OO[x]$
with $\beta\notin F\cup \{\infty\}$.  Then for any
$N\geq 0$ and $k\geq 1$ with $N+k>2$, the $\OO$-points on 
$V(Q)_{N,k}=V(\B)_{N,k}$ are degenerate,
i.e., they are not Zariski dense.
\end{thm}

\begin{proof}
By Theorem \ref{pistol}, any $\OO$-point 
\[
(c_1, \ldots , c_N, c_{N+1},\ldots, c_{N+k})
\text{ with } \left|c_i\right|\geq 2\text{  for } 1\leq i\leq N+k
\]
on the affine variety
$V(\B)_{N,k}$ corresponds to an  $\OO$-PCF
\[
P=[c_1, \ldots , c_{N},\overline{c_{N+1},\ldots, c_{N+k}}]
\]
converging to $\beta$ or $\beta^{\ast}$.
By Theorem \ref{im q}, there is 
at most one way to write $\beta$ or $\beta^{\ast}$ as
$[c_1,\dots,c_{N},\overline{c_{N+1}, \dots, c_{N+k}}]$ with all the
$c_i$ satisfying $c_i\not\in M$ and $\left|c_i\right|\geq 2$.

Now argue as in the proof of Theorem \ref{radishes}.
\end{proof}
\section{Matrix-factorization varieties}
\label{dumber}
Suppose $A\in\SLT(R)$.
\begin{defn}
\label{gen}
{\rm For $A \in \SLT(R)$, let
$\oV_{N,k}(A)\subseteq \A^{N,k}\simeq \A^{N+k}$ be the variety
\[
\oV_{N,k}(A):A=D(y_1) \cdots D(y_N)D(x_1)\cdots D(x_k)tD(-y_N)D(-y_{N-1})
\cdots D(-y_1)t^{k+1}.
\]
}
\end{defn}
\noindent In case $N=0$ we have that $\oV_{0,k}(A)$ is the variety
$\oV_{k}(A)$ of \cite{jz}*{Defn.~1.3(d)}:
\begin{equation}
\label{hardly}
\oV_{0,k}(A)=\oV_{k}(A): A=D(x_1)\cdots D(x_k)t^k
\end{equation}
(using the fact that $t^2 = I_2$).

In case $N=1$, note that $D(a_k)tD(-b_1)=D(a_k-b_1)$,
so that
\begin{align}
\label{pitching}
\oV_{1,k}(A) : A &= D(y_1)D(x_1)\cdots D(x_k)tD(-y_1)t^{k+1}\\
\nonumber &=D(y_1)
D(x_1)\cdots D(x_{k-1})D(x_k-y_1)t^{k+1}.
\end{align}
Hence we have a linear $R$-isomorphism 
\begin{equation}
\label{juice}
\varphi: \oV_{1,k}(A)\rightarrow\oV_{k+1}(A)=\oV_{0, k+1}(A)\mbox{  by  }
(b_1, a_1,\ldots, a_k)\mapsto (b_1, a_1, \ldots , a_{k-1}, a_k-b_1).
\end{equation}

We get a version of Theorems \ref{radishes} and \ref{deg im} for the
varieties $\oV_{N,k}(A)$.

\begin{thm}
Let $\OO$ be $\Z$ or the ring of integers in an imaginary quadratic field.
Suppose $A\in\SLT(\OO)$.  Then for any
$N\geq 0$ and $k\geq 1$ with $N+k>3$, the $\OO$-points on 
$\oV_{N,k}(A)$ are degenerate,
i.e., they are not Zariski dense.
\end{thm}
\begin{proof}
By Theorem \ref{pistol}, any $\OO$-point $p=(b_1, \ldots , b_N,
a_1,\ldots , a_k)$ with $\left|a_i\right|\geq 2$ for $i\geq 1$ on the
affine variety $\oV_{N,k}(A)$ corresponds to a periodic continued
fraction
\[
P(p)=[b_1, \ldots , b_N,\overline{a_1,\ldots, a_k}]
\]
converging to a value $\beta$.  Note that 
$\left(\begin{smallmatrix}\beta\\1 \end{smallmatrix}\right)$ must be an
eigenvector of $At^k$.  By Corollary \ref{at most 1} or Theorem
\ref{im q}, for each eigenvector 
$\left(\begin{smallmatrix}\beta_i\\1 \end{smallmatrix}\right)$
of $At^k$ there is at
most one way to write $\beta_i$ as
$[b_1,\dots,b_N,\overline{a_1,\dots,a_k}]$ with all the
$a_i$ satisfying $\left|a_i\right|\geq 2$ and $a_i\not\in M$ if applicable.
\end{proof}

\section{The fibration of 
\texorpdfstring{$V(\B)_{N,k}$}{V(B)N,k} over the Fermat-Pell rational curve}
\label{radio}
Suppose now that 
$\B=\{\beta,\bs\}$ is the multiset of roots of 
$0\neq Q(x)=Ax^2+Bx+C\in R[x]$.
The Fermat-Pell rational curve $\FP_k(\B)_{/R}=\FP_k(Q)_{/R}$ and the fibration
$\pi_{\FP}:V(\B)_{N,k}\rightarrow \FP_{k}(\B)$ over $R$
are defined in \cite{bej}*{Sect.~3.2}.
The Fermat-Pell curve $\FP_k(\B)\subseteq \A^4_{(a,b,c,d)}$ is defined by
\begin{equation}
\label{dusty}
ad-bc=(-1)^k, \, Bc=A(d-a),\, -Ab=Cc, \, -Bb=C(d-a)\ .
\end{equation}
For a point 
\[
p=(b_1, \ldots, b_N, a_1,\ldots , a_k)\in V(\B)_{N,k}\subseteq
\A^{N,k},
\]
we have 
\begin{equation}
\label{grouch}
\pi_{\FP}(p)= (E_{11}(p), E_{12}(p), E_{21}(p),E_{22}(p))\in \FP_{k}(\B)
\end{equation}
with the polynomials $E_{ij}(b_1, \ldots , b_N, a_1, \ldots , a_k)$,
$1\leq i, j\leq 2$, given in \eqref{gwen}.

In case $\alpha\in R$ but $\sqrt{\alpha}\notin R$, we 
take $Q(x)=x^2-\alpha\in R[x]$ and
write
$\FP_{k}(\alpha):= \FP_{k}(\{\sqrt{\alpha}, -\sqrt{\alpha}\})=\FP_k(Q)$.
In this case
 we have the familiar
plane conic model of the Fermat-Pell rational curve
\begin{equation}
\label{indict}
\FP_{k}(\alpha) : y^2-\alpha x^2=(-1)^k 
\end{equation}
with the $R$-isomorphism between the two models \eqref{dusty}
and \eqref{indict} given by 
\[
(a,b,c,d)\mapsto (c,d)=(x,y) .
\]

If $(a,b,c,d)\in\FP_{k}(\B)$, the  fiber $\pi_{\FP}^{-1}(a,b,c,d)$ 
is defined by the equations
\begin{equation}
\label{punted}
\pi_{\FP}^{-1}(a,b,c,d): E(p)=\left[\begin{smallmatrix}a & b\\c & d\end{smallmatrix}\right].
\end{equation}
Hence we have the following:
\begin{prop}
\label{gadfly}
\textup{(cf. \cite{bej}*{Thm.~3.7(b)}.) } 
The fiber $\pi_{\FP}^{-1}(a, b,c,d)$ is $R$-isomorphic to
$\oV_{N,k}(\left[\begin{smallmatrix}a & b\\c & d\end{smallmatrix}\right] t^k)$.
\end{prop}

We will need to understand when $\#\FP_k(\B)(R)=\infty$.
Because $\FP_k(\B)$ is a rational curve over $R$ with $2$
points at infinity, $\FP_k(\B)(R)$ could be either infinite or finite.
\begin{defn}
\label{pistol1}
Assume $\beta\not\in K\cup \{\infty\}$. Let $R_\beta\subset K(\beta)$
be the (left) order of the $R$-lattice $L_\beta:=\beta R + R$.
\end{defn}

\begin{prop}
\label{pile}
Assume $\beta\not\in K\cup \{\infty\}$ with $\B=\{\beta,\bs\}$
the multiset of roots of $Ax^2+Bx+C\in R[x]$.
\begin{enumerate}[\upshape (a)]
\item
\label{pile1}
$R_{\beta}=R_{\beta^\ast}$.
\item
\label{pile2}
Suppose $A=1$, so $\beta$ and $\bs$ satisfy a {\em monic}
quadratic polynomial over $R$. Then
\[
R_\beta=R_{\beta^\ast}=R[\beta]=R[\bs].
\]
\end{enumerate}
\end{prop}
\begin{proof}
Part (\ref{pile1}) follows from the fact that orders in quadratic
extensions are Galois invariant.  Part (\ref{pile2}) is just the fact
that the (left) order of an order is itself.
\end{proof}

\begin{prop}
\label{fence}
Assume $\beta\not\in K\cup\{\infty\}$.    Then there is a bijective
correspondence between $\FP_k(\B)(R)$ and units $u$  of $R_\beta$
with $\Nm_{K(\beta)/K}(u)= (-1)^k$.
\end{prop}
\begin{proof}
Let $\mathcal{A}=\left[\begin{smallmatrix}a & b\\c &
    d\end{smallmatrix}\right]\in\FP_k(\B)(R)$.  Then the two eigenvalues
of $\mathcal{A}$ are $c\beta +d$ and $c\bs +d$, so
$\det(\mathcal{A})=(c\beta+d)(c\bs+d)=(-1)^k$. Furthermore,
$$(c\beta+d)\cdot1=c\beta+d\in L_\beta$$ and by \eqref{dusty}
$$(c\beta+d)\beta=\left(d-c\frac{B}{A}\right)\beta-c\frac{C}{A}=a\beta+b\in
L_\beta.$$ Therefore, $(c\beta+d)L_\beta\subset L_\beta$ and $u=c\beta+d$ 
is a unit of
$R_\beta$  with $\Nm_{K(\beta)/K}(u)=(-1)^k$.

Conversely, let $c\beta+d\in R_\beta^\times$ have norm $(-1)^k$.  Let
$(c\beta+d)\beta=a\beta+b$.  Then 
\[
\begin{bmatrix}a & b\\c &
    d\end{bmatrix}\in\FP_k(\B)(R).
\]
\end{proof}

\section{Nondegeneracy of integral points on \texorpdfstring{$V(\B)_{1,k}$}{V(B)1,k} and
\texorpdfstring{$V(\B)_{0,k}$}{V(B)0,k}
for \texorpdfstring{$k$}{k} sufficiently large}
\label{pounded}
 We can now formulate and prove our main theorem on the integral
points on the PCF varieties $V(\B)_{1,k}$ and $V(\B)_{0,k}$.
It is convenient to formulate a standard hypothesis:
\begin{hyp}
\label{windyy}
$\HYP_m(R)$ : For every $A\in\SLT(R)$, $\oV_m(A)(R)\neq\emptyset$.
\end{hyp}
\begin{remark}
\label{moules}
{\rm 
\begin{enumerate}[\upshape (a)]
\item
\label{moules1}
 Equivalent formulations of $\HYP_{m}(R)$ are given in
\cite{jz}*{Thm.~2.2(b)}.
\item
\label{moules2}
For $R$ with $R^\times$ infinite,
the hypothesis $\HYP_{m}(R)$ implies that $\oV_{m}(A)(R)$ is 
Zariski dense if $m\geq 4$ for every $A\in\SLT(R)$ by 
\cite{jz}*{Thm.~4.4(b)}.
\end{enumerate}
}
\end{remark}

\begin{thm}
\label{main}
Let
$\B=\{\beta,\bs\}$ be the multiset of roots of 
$Ax^2+Bx+C\in R[x]$ with $\beta\notin K\cup\{\infty\}$.
Suppose there are infinitely many
units in $R_\beta$ as in Definition \textup{\ref{pistol1}} with  norm 
down to  $K$ equal to
$(-1)^k$ and suppose
that there are infinitely many units in $R$.
\begin{enumerate}[\upshape (a)]
\item
\label{main1}
\begin{enumerate}[\upshape (i)]
\item
\label{main11}
Assume $\HYP_{m}(R)$. Then the $R$-points
of $V(\B)_{1,k}$ are Zariski dense if $k\geq m-1\geq 3$.
\item
\label{main12}
Assume $\HYP_{m}(R)$.  Then the $R$-points of $V(\B)_{0,k}$ are Zariski dense 
if $k\geq m\geq 4$.
\end{enumerate}
\item
\label{main2}
\begin{enumerate}[\upshape (i)]
\item
\label{main21}
 The $R$-points of $V(\B)_{1,k}$ are Zariski dense if $k\geq 8$. 
\item
\label{main22}
The $R$-points of  $V(\B)_{0,k}$ are Zariski dense
if $k\geq 9$.  
\end{enumerate}
\item
\label{main3}
Assume the Generalized Riemann Hypothesis as 
in \cite{jz1}*{Sect.~3D}. 
\begin{enumerate}[\upshape (i)]
\item
\label{main31}
The $R$-points on  $V(\B)_{1,k}$ are
Zariski dense for $k\geq 4$ if $S$ contains a real archimedean prime, 
for $k\geq 5$ if $S$
contains a finite prime, and for $k\geq 6$ in general. 
\item
\label{main32}
The $R$-points on  $V(\B)_{0,k}$ are
Zariski dense for $k\geq 5$ if $S$ contains a real archimedean prime, 
for $k\geq 6$ if $S$
contains a finite prime, and for $k\geq 7$ in general. 
\end{enumerate}
\end{enumerate}
\end{thm} 
\begin{proof}
The assumption that there are infinitely many units
in $R_\beta$ with relative norm $(-1)^k$ is simply the 
statement that the Fermat-Pell rational curve $\FP_{k}(\B)$ has infinitely
many $R$-points by Proposition \ref{fence}  .  
For $\pi_{\FP}:V(\B)_{1,k}\rightarrow \FP_{k}(\B)$ and
an $R$-point $(a,b,c,d)\in \FP_{k}(\B)$, 
$\pi_{\FP}^{-1}(a,b,c,d))\subseteq V(\B )_{1,k}$
is $R$-isomorphic to the variety $\oV_{1,k}(\left[\begin{smallmatrix} a &b\\
c & d\end{smallmatrix}\right]t^k)\simeq
\oV_{k+1}(\left[\begin{smallmatrix} a &b\\
c & d\end{smallmatrix}\right]    t^k)$ by
Proposition \ref{gadfly} and \eqref{juice}.
Similarly the fiber $\pi_{\FP}^{-1}(a,b,c,d)$
of the Fermat-Pell fibration for $V(\B)_{0,k}$
is $\oV_{k}(\left[\begin{smallmatrix} a &b\\
c & d\end{smallmatrix}\right] t^k)$.

Assume $\HYP_{m}(R)$ for $m\geq 4$; then $\oV_{m}(A)(R)$ is Zariski dense for any
$A\in \SLT(R)$ by Remark
\ref{moules}(\ref{moules2}).
Apply these statements with $A= \left[\begin{smallmatrix} a &b\\
c & d\end{smallmatrix}\right] t^k$ to show that the 
$R$-points are Zariski dense on the fibers of the Fermat-Pell
fibrations under the indicated hypotheses. But now
if we have fibrations
\begin{equation}
\label{wren}
\pi_{\FP}:V(\B)_{1,k}\rightarrow \FP_{k}(\B), \quad 
\pi_{\FP}:V(\B)_{0,k}\rightarrow \FP_{k}(\B)
\end{equation}
such that the set of fibers whose $R$-points are Zariski dense is a
Zariski dense subset of the base, then the $R$-points 
are Zariski dense on the total space, thereby proving Theorem
\ref{main}(\ref{main1}).

To prove \eqref{main2}, note that $\HYP_{m}(R)$ is true for $m\geq 9$ by
\cite{mrs}*{Thm.~1.1}. To prove \eqref{main3} assuming the Generalized Riemann
Hypothesis \cite{jz1}*{Sect.~3D}, note that by
\cite{jz1}*{Thm.~1.3}, $\HYP_m(R)$ is true 
 for $m\geq 5$ if $S$ contains a real archimedean prime, 
for $m\geq 6$ if $S$
contains a finite prime, and for $m\geq 7$ in general.
\end{proof}
\begin{remark}
{\rm
If $S$ contains a real archimedean prime or a finite prime, then
under the assumptions of Theorem \ref{main} one can conclude by
\cite{jz}*{Thm.~1.3(b)}
that $V(\B)_{1,k}(R)$ is Zariski dense for $k\geq 7$
using  and $V(\B)_{0,k}(R)$ is Zariski dense
for $k\geq 8$.
}
\end{remark}
We also record a theorem on the geometry of PCF varieties:
\begin{thm}
\label{queen}
\begin{enumerate}[\upshape (a)]
\item
\label{queen1}
Suppose $k\geq 5$.  Then
$V(\B)_{1,k}$ is an irreducible rational variety
 over the
field $K$. 
\item
\label{queen2}
Suppose $k\geq 6$.  Then $V(\B)_{0,k}$ is an irreducible
rational variety over the field $K$.
\end{enumerate}
\end{thm}
\begin{proof}
If $k\geq 5$, we have by Proposition \ref{gadfly} 
and \cite[Thm.~3.3(d)]{jz} that the fiber 
$\pi_{\FP}^{-1}(a,b,c,d)$ of the Fermat-Pell fibration
$\pi_{\FP}:V(\B)_{1,k}\rightarrow \FP_{k}(\B)$
is rational with a birational equivalence $g_k$ to 
$\A^{k-2}$ given by
$(x_1,\ldots,x_{k+1})\mapsto(x_4,\ldots,x_{k+1})$. 
Hence
\[ (g_k,\pi_{\FP}):V(\B)_{1,k}\rightarrow \A^{k-2}\times \FP_k(\B)
\]
is a birational
equivalence.  

The same argument applies to $V(\B)_{0,k}$ for $k\geq 6$.
\end{proof}
\begin{remark}
\label{pizza}
{\rm
\begin{enumerate}[\upshape (a)]
\item
The possibilities for $V(\B)_{1,k}$ for $1\leq k\leq 4$ are as follows:
\begin{enumerate}[\upshape (i)]
\item
$V(\B)_{1,1}$ is generically $2$ points and otherwise
empty by \cite[Sect.~5.3]{bej}.
\item
$V(\B)_{1,2}$ is generically the union of 
$2$ irreducible rational curves and otherwise $1$ irreducible rational curve, 
the union of $3$ irreducible rational curves,
or an irreducible rational surface by  \cite[Sect.~8]{bej}.
\item
$V(\B)_{1,3}$ is an irreducible
 rational surface by the same proof as Theorem \ref{queen},
with an additional argument that the generic fiber can't be the reducible
Case II in \cite[Thm. 3.3(b)]{jz}.
\item
$V(\B)_{1,4}$ is either an irreducible rational threefold or the union of two 
irreducible rational threefolds.  Again the proof follows Theorem \ref{queen}
if the fiber of $\pi_{\FP}$ is generically irreducible;
here $V(\B)_{1,4}$ is a rational threefold.
However, we can have $\B$ with  associated $Q(x)=Ax^2+Bx+C$
such that the generic fiber is
in the reducible Case III of \cite[Thm.~3.3(c)]{jz},
in which case $V(\B)_{1,4}$ is the union of two irreducible rational threefolds.
\end{enumerate}
\item
The possibilities for $V(\B)_{0,k}$ 
for $1\leq k \leq 5$ are as follows:
\begin{enumerate}
\item
$V(\B)_{0,1}$ is generically empty and otherwise one point
by \cite[Sect.~5.1]{bej}. 
\item
$V(\B)_{0,2}$ is generically $2$ points and otherwise one point or
one line by \cite[Sect.~5.2]{bej}.
\item
 The equations for $V(\B)_{0,3}$ 
are given in \cite[(20)]{bej}.  Generically
$V(\B)_{0,3}$ is an irreducible rational curve \cite[Sect.~6]{bej}.
Examination of the degenerate cases gives that otherwise
$V(\B)_{0,3}$ is either the union of $2$ irreducible rational curves or 
the union of 
$3$ irreducible rational curves.
\item
$V(\B)_{0,4}$ either an irreducible  rational surface or the union of two
irreducible rational surfaces. If the generic fiber of $\pi_{\FP}$ is not
in Case II of \cite[Thm.~3,3(b)]{jz}, then it is an irreducible rational curve
and $V(\B)_{0,4}$ is an irreduble 
 rational surface as in Theorem \ref{queen}.  If the 
generic fiber of $\pi_{\FP}$ is in Case II of \cite[Thm.~3.3(b)]{jz},
then it is a union of two irreducible rational curves and $V(\B)_{0,4}$ is the
union of two irreducible rational surfaces.
\item
$V(\B)_{0,5}$ is an irreducible  rational threefold following the proof of
Theorem \ref{queen} after arguing that the generic fiber of $\pi_{\FP}$
 cannot
be the reducible Case III of \cite[Thm.~3.3(c)]{jz}.
\end{enumerate}
\end{enumerate}
}
\end{remark}
\noindent Note that not all the varieties $V(\B)_{N,k}$ are rational:
for example it is shown in \cite{bej}*{Sect.~7} that 
$V(\B)_{2,1}$ is generically a curve of genus $1$.

\bibliographystyle{plain}
\bibliography{ZCa-adam}%../local,outside
\end{document}